\documentclass[a4paper]{amsart}
\usepackage{amsmath,amssymb,amsxtra,amsfonts,amsthm,url}

\usepackage{tikz}
\usetikzlibrary{arrows,backgrounds,decorations,automata}

\theoremstyle{plain}
\newtheorem{theorem}{Theorem}
\newtheorem{lemma}[theorem]{Lemma}

\newtheorem{proposition}[theorem]{Proposition}
\newtheorem{conjecture}[theorem]{Conjecture}

\theoremstyle{definition}

\begin{document}


\title[Mind the gap]{The minimal growth of a $k$-regular sequence}

\author{Jason P. Bell}
\address{Department of Pure Mathematics, University of Waterloo, Waterloo, Canada}
\email{jpbell@uwaterloo.ca}

\author{Michael Coons}
\address{School of Math.~and Phys.~Sciences\\
University of Newcastle\\
Callaghan\\
Australia}
\email{Michael.Coons@newcastle.edu.au}

\author{Kevin G. Hare}
\address{Department of Pure Mathematics, University of Waterloo, Waterloo, Canada}
\email{kghare@uwaterloo.ca}
\date{\today}

\thanks{Research of J.~P.~Bell was supported by NSERC grant 31-611456, the research of M.~Coons was supported by ARC grant DE140100223, and the research of K.~G.~Hare was partially supported by NSERC.}

\begin{abstract} We determine a lower gap property for the growth of an unbounded \(\mathbb{Z}\)-valued \(k\)-regular sequence. In particular, if \(f:\mathbb{N}\to\mathbb{Z}\) is an unbounded \(k\)-regular sequence, we show that there is a constant \(c>0\) such that \(|f(n)|>c\log n\) infinitely often. We end our paper by answering a question of Borwein, Choi, and Coons on the sums of completely multiplicative automatic functions.
\end{abstract}


\maketitle

\section{Introduction}

Let $k\geqslant 2$ be a natural number and $\Sigma$ a finite set of integers. A sequence $f:\mathbb{N}\to \Sigma$ is called {\em $k$-automatic} (or just {\em automatic}), provided that there is a finite-state automaton that reads as input the base-$k$ expansion of $n$ and outputs the number $f(n)$. The canonical example of an automatic sequence is the Thue-Morse sequence $$\{t(n)\}_{n\geqslant 0}:=0110100110010110100101100110\cdots,$$ wherein $t(n)$ takes the value $1$ if the binary expansion of $n$ has an odd number of ones, and the value $0$ if the binary expansion has an even number of ones. The automaton that takes in the binary expansion of $n$ and outputs $t(n)$ is given in Figure \ref{TM}.

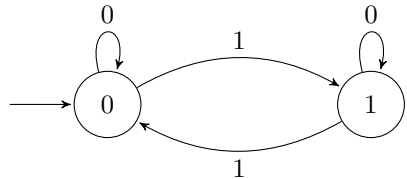
\begin{figure}[htb]
\begin{center}
\hspace{-.7cm}
\begin{tikzpicture}[scale=.77, ->,>=stealth',shorten >=1pt,auto,scale=1.5]

	\node [state,draw=none,overlay] (q) at (-.5,1) {};
	\node [state] (0) at (1,1) {$0$};
	\node [state] (1) at (4,1) {$1$};
	
	\path (q) edge [right] node {} (0);
	
	\path (0) edge [loop above] node {$0$} (0);
	\path (1) edge [loop above] node {$0$} (1);
	
	\path (0) edge [bend left] node {$1$} (1);
	\path (1) edge [bend left] node {$1$} (0);

\end{tikzpicture}
\caption{The $2$-automaton that produces the Thue-Morse sequence.}
\label{TM}
\end{center}
\end{figure}

We define the {\em $k$-kernel} of an integer-valued sequence $f:\mathbb{N}\to \mathbb{Z}$ to be the collection of subsequences $$\mathcal{K}_k(f):=\big\{\{f(k^\ell n+r)\}_{n\geqslant 0}\colon \ell\geqslant0, 0\leqslant r<k^\ell\big\}.$$ For example, the $2$-kernel of the Thue-Morse sequence is $$\mathcal{K}_2(t)=\big\{\{t(n)\}_{n\geqslant 0},\{1-t(n)\}_{n\geqslant 0}\big\}.$$ Indeed, a result of Cobham \cite{C1972} gives that $\mathcal{K}_k(f)$ is finite if and only if $\{f(n)\}_{n\geqslant 0}$ is a $k$-automatic sequence; see also \cite[Theorem 6.6.2]{AS2003}. Using this characterisation of $k$-automatic sequences, Allouche and Shallit \cite{AS1992} introduced the more general class of $k$-regular sequences. Formally, we say the integer-valued sequence $f:\mathbb{N}\to \mathbb{Z}$ is {\em $k$-regular} provided the $\mathbb{Z}$-module generated by $\mathcal{K}_k(f)$ (regarded as a submodule of the $\mathbb{Z}$-module of all integer-valued sequences) is finitely generated. Of course, if $\mathcal{K}_k(f)$ is finite, then certainly such a $\mathbb{Z}$-module is finitely generated, so a $k$-automatic sequence is necessarily $k$-regular.  Moreover, Allouche and Shallit \cite[Theorem 2.3]{AS1992} showed that a $k$-regular sequence taking only finitely many distinct values is necessarily $k$-automatic. The following question arises immediately: {\em what type of growth must an unbounded $k$-regular sequence have?}

To answer this question, in this paper we provide the following result.

\begin{theorem}\label{main} Let $k\geqslant 2$.  If $f:\mathbb{N}\to \mathbb{Z}$ is an unbounded $k$-regular sequence, then there exists $c>0$ such that $|f(n)| > c \log n$ infinitely often.
\end{theorem}

\noindent In fact, we can show that there exist words $u_1,\ldots ,u_m,y,v_1,\ldots ,v_m\in\{0,1,\ldots ,k-1\}^*$ and a constant $c_0>0$ such that for all sufficiently large $n$ there exist an $i$ and $j$ such that $|f([u_i y^n v_j]_k)|\geqslant c_0 n$. Here for a word $w=i_s\cdots i_0\in \{0,1,\ldots ,k-1\}^*$, we have written $[w]_k=i_sk^s+\cdots +i_0$. This can be thought of as a type of ``pumping lemma'' for attaining unbounded growth.

Before moving on, we present examples that show that the conclusion to the statement of Theorem \ref{main} is best possible.

Let $k\geqslant 2$ and consider the sequence $s_k(n)$, which is the sum of the digits of the base-$k$ expansion of $n$.  We note that if $a\geqslant 1$ and $0\leqslant b <k^a$ then $s_k(k^a n + b)=s_k(n)+s_k(b)$ and so the $\mathbb{Z}$-module spanned by the $k$-kernel of $s_k(n)$ is generated by the sequence $s_k(n)$ and the constant sequence of ones, and so $s_k(n)$ is $k$-regular.  Since the base-$k$ expansion of $n$ has at most $\log_k(n) +1$ digits, each of which is at most $k-1$, we see that there is a positive constant $C$ such that $s_k(n)\leqslant C\log(n)$ for $n$ sufficiently large.  On the other hand, $s_k(k^n)=1$ for every $n$ and thus a general answer to the growth question can at most be an ``infinitely often'' result.   If fact, if one lets $t_k(n)$ denote the $k$-automatic sequence that is $1$ when $n=k^j-1$ for some $j\geqslant 0$ and is zero otherwise, then the product $u_k(n):=s_k(n)t_k(n)$ is $k$-regular \cite[Theorem 16.2.1]{AS2003} and we have $u_k(n)=0$ unless $n=k^j-1$, in which case it is $(k-1)j$, and so we see that the natural numbers $n$ for which $|u_k(n)|>\log n$ are precisely those $n$ of the form $n=[(k-1)^j]_k$ for some $k\geqslant 1$.  

\section{Minimal growth in unbounded regular sequences}
In this section, we prove our main result.  To do this, we need a basic dichotomy about the growth of entries of matrices in a semigroup.

\begin{lemma} Let $k\geqslant 2$ be an integer, let ${\bf A}_0,\ldots,{\bf A}_{k-1}$ be $d\times d$ integer matrices, and let $\mathcal{B}$ be the semigroup generated by ${\bf A}_0,\ldots ,{\bf A}_{k-1}$.  Then either $\mathcal{B}$ is finite or there is some ${\bf S}\in \mathcal{B}$ and fixed vectors ${\bf v}$ and ${\bf w}\in \mathbb{C}^d$ such that $|{\bf w}^T {\bf S}^n {\bf v}| \geqslant n$ for all sufficiently large $n$.
\label{lem: growthlem}
\end{lemma}
\begin{proof}
Suppose that $\mathcal{B}$ is infinite.  Then since $\mathcal{B}$ is finitely generated, a result of McNaughton and Zalcstein \cite{MZ1975} gives that there is some ${\bf S}$ in $\mathcal{B}$ such that the matrices ${\bf S},{\bf S}^2,{\bf S}^3, \ldots $ are all distinct.  Let $p(x)$ be the characteristic polynomial of ${\bf S}$.  Then $p(x)$ is a monic integer polynomial.  If $p(x)$ has a root $\lambda$ that is strictly greater than $1$ in modulus, then ${\bf S}$ has an eigenvector ${\bf v}$ such that ${\bf S}{\bf v}=\lambda {\bf v}$.  Pick a nonzero vector ${\bf w}$ such that ${\bf w}^T {\bf v}=C\neq 0$.  Then $|{\bf w}^T {\bf S}^n {\bf v}| = |C|\cdot  |\lambda|^n \geqslant n$ for $n$ sufficiently large.  

If, on the other hand, all the roots of $p(x)$ are at most $1$ in modulus, then all non-zero eigenvalues of ${\bf S}$ are  algebraic integers with all conjugates having modulus 1, hence they are roots of unity. Let ${\bf Y}$ be a matrix in ${\rm GL}_d(\mathbb{C})$ such that ${\bf T}:={\bf Y}^{-1}{\bf S}{\bf Y}$ is in Jordan form, where we take Jordan blocks to be upper-triangular.  Then each Jordan block in ${\bf T}$ is of the form ${\bf J}_i(\lambda)$ with $\lambda$ either zero or a root of unity and $i\geqslant 1$.  Since ${\bf S}$ does not generate a finite subsemigroup of $\mathcal{B}$, there is some root of unity $\omega$ and some $m>1$ such that ${\bf T}$ has a block of the form ${\bf J}_m(\omega)$.  We may assume, without loss of generality, that ${\bf J}_m(\omega)$ is the first block occurring in ${\bf T}$.  Then the $(1,2)$-entry of
${\bf T}^n$ is $n\omega^{n-1}$ and so $|e_1^T {\bf T}^n e_2|=n$ for every $n$.  In particular, we have
$$|e_1^T{\bf Y}^{-1}{\bf S}^n{\bf Y}e_2|\geqslant n$$ for every $n$.  Taking ${\bf w}^T=e_1^T{\bf Y}^{-1}$ and ${\bf v}={\bf Y}e_2$ gives the result.
\end{proof}
We are now ready to prove Theorem \ref{main}. 

\begin{proof}[Proof of Theorem \ref{main}] Let $k\geqslant 2$ be an integer, and suppose that $f:\mathbb{N}\to \mathbb{Z}$ is an unbounded $k$-regular sequence.  Given a word $w=i_s\cdots i_0\in \{0,\ldots ,k-1\}^*$, as stated previously, we let $[w]_k$ denote the natural number $n=i_s k^s+\cdots +i_1 k+i_0$.  We have the $\mathbb{Z}$-submodule of all $\mathbb{Z}$-valued sequences spanned by $\mathcal{K}_k(f)$ is a finitely generated torsion free module and hence free of finite rank.  Let $\big\{\{g_1(n)\}_{n\geqslant 0},\ldots ,\{g_d(n)\}_{n\geqslant 0}\big\}$ be a $\mathbb{Z}$-module basis for the $\mathbb{Z}$-module spanned by $\mathcal{K}_k(f)$.  Then for each $i\in \{0,1,\ldots ,k-1\}$, the functions $g_1(kn+i),\ldots ,g_d(kn+i)$ can be expressed as $\mathbb{Z}$-linear combinations of $g_1(n),\ldots ,g_d(n)$ and hence there are $d\times d$ integer matrices ${\bf A}_0,\ldots ,{\bf A}_{k-1}$ such that 
$$ [g_1(n),\ldots ,g_d(n)]{\bf A}_i =[g_1(kn+i),\ldots ,g_d(kn+i)]$$ for $i=0,\ldots ,k-1$ and all $n\geqslant 0$.  In particular, if $i_s\cdots i_0$ is the base-$k$ expansion of $n$, then $ [g_1(0),\ldots ,g_d(0)]{\bf A}_{i_s}\cdots {\bf A}_{i_0} = [g_1(n),\ldots ,g_d(n)]$.   (We note that this holds even if we pad the base-$k$ expansion of $n$ with zeros at the beginning.)  We claim that the $\mathbb{Q}$-span of the vectors $[g_1(i),\ldots ,g_d(i)]^T$, as $i$ ranges over all natural numbers, must span all of $\mathbb{Q}^d$.  Indeed, if this were not the case, then their span would be a proper subspace of $\mathbb{Q}^d$ and hence the span would have a non-trivial orthogonal complement.  In particular, there would exist integers $c_1,\ldots ,c_d$, not all zero, such that $$c_1g_1(n)+\cdots +c_dg_d(n)=0$$ for every $n$, contradicting the fact that $g_1(n),\ldots ,g_d(n)$ are linearly independent sequences. 

Let $\mathcal{A}$ denote the semigroup generated by ${\bf A}_0,\ldots ,{\bf A}_{k-1}$.  Then we have just shown that there exist words ${\bf X}_1,\ldots ,{\bf X}_d$ in $\mathcal{A}$ such that $$  [g_1(0),\ldots ,g_d(0)]{\bf X}_1 ,\ldots , [g_1(0),\ldots ,g_d(0)]{\bf X}_d$$ span $\mathbb{Q}^d$.  
Now, if $\mathcal{A}$ is finite, then $\{g_1(n)\}_{n\geqslant 0},\ldots ,\{g_d(n)\}_{n\geqslant 0}$ take only finitely many distinct values.  Since $\{f(n)\}_{n\geqslant 0}$ is a $\mathbb{Z}$-linear combination of $\{g_1(n)\}_{n\geq 0},\ldots ,\{g_d(n)\}_{n\geqslant 0}$, we see that it too takes only finitely many distinct values, which contradicts our assumption that it is unbounded.  Thus $\mathcal{A}$ must be infinite.  By Lemma \ref{lem: growthlem}, there exist ${\bf Y}\in \mathcal{A}$ and vectors ${\bf x},{\bf y}\in \mathbb{C}^d$ such that $|{\bf x}^T {\bf Y}^n {\bf y}|\geqslant n$ for all $n$ sufficiently large.

By construction, we may write ${\bf x}^T=\sum_j \alpha_j {\bf X}_j [g_1(0),\ldots ,g_d(0)]$ for some complex numbers $\alpha_j$.  
Then 
$${\bf x}^T {\bf Y}^n = \sum_{j} \alpha_j [g_1(0),\ldots ,g_d(0)]{\bf X}_j {\bf Y}^n .$$     Let $u_j$ be the word in $\{0,1,\dots ,k-1\}^*$ corresponding to ${\bf X}_j$ and let $y$ be the word in $\{0,\ldots ,k-1\}^*$ corresponding to ${\bf Y}$; that is $u_j=i_s\cdots i_0$ where ${\bf X}_j = {\bf A}_{i_s}\cdots {\bf A}_{i_0}$ and similarly for $y$. Then we have $$[g_1(0),\ldots ,g_d(0)]{\bf X}_j{\bf Y}^n=[g_1([u_j y^n]_k),\ldots ,g_d([u_j y^n]_k)]^T.$$

Write ${\bf y}^T=[\beta_1,\ldots ,\beta_d]$.  Then $${\bf x}^T {\bf Y}^n {\bf y} =\sum_{i,j} \alpha_i\beta_j g_j([u_i y^n]_k).$$
By assumption, each of $\{g_1(n)\}_{n\geqslant 0},\ldots ,\{g_d(n)\}_{n\geqslant 0}$ is in the $\mathbb{Z}$-module generated $\mathcal{K}_k(f)$, and hence there exist natural numbers $p_1,\ldots ,p_t$ and $q_1,\ldots ,q_t$ with $0\leqslant q_m <k^{p_m}$ for $m=1,\ldots ,t$ such that each of for $s=1,\ldots ,d$ we have
$g_s(n)= \sum_{i=1}^t  \gamma_{i,s} f(k^{p_i} n + q_i)$ for some integers $\gamma_{i,s}$.  
Then 
$${\bf x}^T{\bf Y}^n {\bf y} = \sum_{i,j,\ell} \alpha_i \beta_j \gamma_{\ell,j} f([u_i y^n v_{\ell}]_k),$$ where $v_\ell$ is the unique word in $\{0,1,\ldots ,k-1\}^*$ of length $p_\ell$ such that $[v_\ell]_k=q_\ell$.  Let $K=\sum_{i,j,\ell} |\alpha_i|\cdot |\beta_j|\cdot |\gamma_{\ell,j}|$.  Then since
$|{\bf x}^T {\bf Y}^n {\bf y}|\geqslant n$ for all $n$ sufficiently large, there is some $N_0>0$ such that for $n>N_0$ some element from
$$\big\{ \{|f([u_i y^n v_j]_k)|\}_{n\geqslant 0} \colon i=1,\ldots, d, j=1,\ldots ,t\}\big\}$$ is at least $n/K$.  

We let $M$ denote the maximum of the lengths of $u_1,\ldots ,u_d,y,v_1,\ldots ,v_t$.
Then each of $[u_i y^n v_j]_k < k^{2Mn}$ for $n\geqslant 2$.   Hence we have constructed an infinite set of natural numbers $N=N_n:=[u_i y^n v_j]_k$ such that 
$|f(N)|>\log_k(N)/2K$ and so taking $c=(2MK\log\, k)^{-1}$, we see that $|f(N)|>c\log N$ for infinitely many $N$.
\end{proof}

\section{Discrepancy of completely multiplicative automatic functions}

Let us take this opportunity to highlight one of our favourite open problems, the Erd\H{o}s discrepancy problem, which may or may not be related to the results above (only time will tell).

Erd\H{o}s \cite{E1957} asked the following question: {\em ``Let $f(n)=\pm 1$ be an arbitrary number theoretic function. Is it true that to every $c$ there is a $d$ and an $n$ for which \begin{equation}\label{E8}\left|\sum_{j=1}^n f(jd)\right|>c?\end{equation} Inequality \eqref{E8} is one of my oldest conjectures.''} (This particular quote is taken from a restatement of the conjecture in \cite[p.~78]{E1985}. See also \cite{E1985b} and \cite{EG}.) Erd\H{o}s offers 500 dollars for a proof of this conjecture. Erd\H{o}s \cite[p.~293]{E1957} wrote in 1957 that this conjecture is twenty-five years old, placing its origin at least as far back as the early 1930s. Erd\H{o}s \cite{E1957,E1985,E1985b} also states the multiplicative form of his conjecture. 

\begin{conjecture}[Erd\H{o}s]\label{Econj} Let $f(n)=\pm 1$ be a multiplicative function, i.e., $f(ab)=f(a)f(b)$, when $(a,b)=1$. Then  \begin{equation}\label{E9}\limsup\left|\sum_{j=1}^n f(j)\right|=\infty.\end{equation}\end{conjecture} 

\noindent Erd\H{o}s added in \cite{E1985} that {\em ``clearly \eqref{E9} would follow from \eqref{E8} but as far as I know \eqref{E9} has never been proved. Incidentally \eqref{E9} was also conjectured by Tchudakoff.''}

Erd\H{o}s \cite{E1985} suggests a strengthening of his conjecture: {\em ``is it true that the density of integers $n$ for which $\left|\sum_{j=1}^n f(j)\right|<c$ is $0$ for every $c$?''}

Concerning true lower bounds to $\sum_{n\leqslant x}f(n)$, Erd\H{o}s and Graham \cite{EG} ask, is it true that there is a $c>0$ so that $$\max_{N\leqslant x}\left|\sum_{n\leqslant N} f(n)\right|>c\log x?$$ 

We note that many of the multiplicative functions defined in terms of Dirichlet characters are automatic and since the sequence of partial sums of an automatic sequence is regular, Theorem \ref{main} shows that this lower bound must hold for such multiplicative functions as long as their partial sums are unbounded.  As an example, consider the completely multiplicative function $\lambda_3(n)$ defined by $\lambda_3(1)=1$ and $$\lambda_3(p)=\begin{cases}~\,\, 1 & \mbox{if $p\equiv 1\ (\bmod\ 3)$}\\ -1 & \mbox{if $p\equiv -1\ (\bmod\ 3)$}\\ ~\,\, 1 & \mbox{if $p=3$.}\\ \end{cases}$$ for all primes $p$. To see that $\{\lambda_3(n)\}_{n\geqslant 0}$ is $3$-automatic one can either show that $\lambda_3(3n+1)=1$, $\lambda_3(3n+2)=-1$ and $\lambda_3(3n)=\lambda_3(n)$ and use Cobham's result characterizing $k$-automaticity in terms of finiteness of the $k$-kernel, or one can see Figure \ref{g3}.

\begin{figure}[htb]
\begin{center}
\hspace{-.7cm}
\begin{tikzpicture}[scale=.77, ->,>=stealth',shorten >=1pt,auto,scale=1.5]

	\node [state,draw=none,overlay] (q) at (-.5,1) {};
	\node [state] (0) at (1,1) {$1$};
	\node [state] (1) at (4,1) {$-1$};
	
	\path (q) edge [right] node {} (0);
	
	\path (0) edge [loop above] node {$0$} (0);
	\path (1) edge [loop above] node {$0$} (1);
	\path (0) edge [loop below] node {$1$} (0);
	\path (1) edge [loop below] node {$2$} (1);
		
	\path (0) edge [bend left] node {$2$} (1);
	\path (1) edge [bend left] node {$1$} (0);

\end{tikzpicture}
\caption{The $3$-automaton that produces the sequence $\lambda_3(n)$.}
\label{g3}
\end{center}
\end{figure}
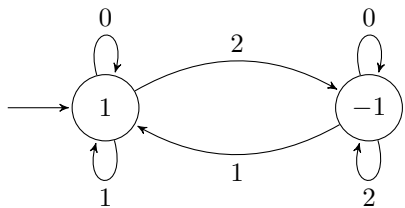

We have the somewhat surprising result that $s_{1,3}(n)=\sum_{j=1}^n \lambda_3(j)$, where $s_{1,3}(n)$ is the number of ones in the ternary expansion of $n$.  To see this, let $F(n)=\sum_{j=0}^n \lambda_3(j)$, where we take $\lambda_3(0)=0$.  Then we have
$$F(3n)=\sum_{j=0}^n \lambda_3(3j)+\sum_{j=0}^n \lambda_3(3j+1)+\sum_{j=0}^n \lambda_3(3j+2)=F(n),$$ 
$$F(3n+1)=\sum_{j=0}^n \lambda_3(3j)+\sum_{j=0}^n \lambda_3(3j+1)+\sum_{j=0}^{n-1} \lambda_3(3j+2)=F(n)+1,$$ and
$$ F(3n+2)=\sum_{j=0}^n \lambda_3(3j)+\sum_{j=0}^n \lambda_3(3j+1)+\sum_{j=0}^{n-1} \lambda_3(3j+2)=F(n).$$  On the other hand, it is clear that 
$$s_{1,3}(3n)=s_{1,3}(3n+2)=s_{1,3}(n)$$ and $s_{1,3}(3n+1)=s_{1,3}(n)+1$.   Since $s_{1,3}(0)=F(0)$ and they satisfy the same recurrences given above, we see that $F(n)=s_{1,3}(n)$ for every $n$.   In particular, $F((3^n-1)/2)=n$ for every $n$, and so in view of the function $\lambda_3(n)$ and its sequence of partial sums, $F(n)$, the lower bounds Erd\H{o}s and Graham \cite{EG} speculated might hold would indeed be ``best possible.''

Functions such as $\lambda_3(n)$ are defined via Dirichlet characters and as well are automatic. As stated previously, their partial sums are regular, and thus our theorem applies to them. Understanding and relating multiplicative functions to character-like functions, as a means to attack the Erd\H{o}s discrepancy problem, has been described by Gowers \cite{Gedp}. Specifically, Gowers \cite{Gedp} puts forward the following strategy for a possible proof: 
\begin{quote}
{\em
\begin{enumerate}
\item[{(a)}] Develop a notion of ``almost character-like'' functions.
\item[{(b)}] Prove that a multiplicative sequence that is not almost character-like must have unbounded discrepancy.
\item[{(c)}] Prove that a multiplicative sequence that is almost character-like must have unbounded discrepancy.
\end{enumerate}
}
\end{quote}
\noindent He also adds, {\em ``I would expect {\rm (c)} to be easy once one had a decent answer to {\rm (a)}. The real challenge is to answer {\rm (a)} in a way that allows one to prove a good theorem that answers {\rm (b)}.''} 

Schlage-Puchta \cite[Proposition 1]{SP2011} has given the following characterisation of nonvanishing $k$-automatic completely multiplicative functions, which, with a bit of work, allows one to show that the question of Erd\H{o}s and Graham \cite{EG} holds for completely multiplicative automatic functions.

\begin{proposition}[Schlage-Puchta \cite{SP2011}]\label{SPf} Let $k\geqslant 2$ be an integer, and $f$ be a completely multiplicative $k$-automatic function, which does not vanish. Then there exists and integer $r$, such that if $n_1,n_2,\ell$ are integers such that $\gcd(n_1,k^{\ell+1})$ divides $k^\ell$, and $n_1\equiv n_2\  (\bmod\ k^{r+\ell}),$ then $f(n_1)=f(n_2).$
\end{proposition}

Proposition \ref{SPf} immediately implies that a nonvanishing completely multiplicative $k$-automatic function $f$ agrees with a Dirichlet character $\chi$ modulo $k^\ell$ on a all values coprime to $k$; that is, $f(n)=\chi(n)$ for all $n$ with $\gcd(n,k)=1$. In fact, one can say a bit more about these functions. 

\begin{proposition}\label{q} Let $k\geqslant 2$ be an integer, and $f$ be a completely multiplicative $k$-automatic function, which does not vanish. Then $k=q^m$ for some prime $q$ and some integer $m\geqslant 1$.
\end{proposition}

\begin{proof} Write $k=q_1^{j_1}\cdots q_s^{j_s}$. We will show that if $s>1$, then $f$ cannot be $k$-automatic unless it is eventually periodic. 

To this end, let $s>1$, and set $l= k/q_1$.  Then for any given $a\geqslant 1$ and $0\leqslant b<l^a$, $$\{f(l^a n + b)\}_{n\geq 0} = \{f(q_1)^{-a} f(k^a n + b q_1^a)\}_{n\geq 0} \in \zeta\cdot \mathcal{K}_k(f),$$ where $\zeta$ is some root of unity. Since $f$ is $k$-automatic, we see it is $l$-automatic too.  But $k$ and $l$ are multiplicatively independent since $s>1$, so Cobham's theorem \cite{C1969} gives the result.  
\end{proof}

Relating this to the question of Erd\H{o}s, let us now suppose that $f(n)=\pm 1$ is a $k$-automatic completely multiplicative function. To examine the growth of the summatory function, Proposition \ref{q} gives that we need only consider $k=q^m$ for some prime $q$. Since $f$ is completely multiplicative and real, Proposition \ref{SPf} gives that $f$ agrees with a real Dirichlet character modulo $q^m$ on values coprime to $q$. This leaves only two options for the Dirichlet character, it is either the trivial character or the quadratic character, and moreover, these options allow us to work modulo $q$ (independent of the value of $m$).

Set $$F(x) := \sum_{\substack{1\leqslant n \leqslant x\\ (n,q)=1}} f(n)\quad\mbox{and}\quad G(x):=\sum_{n\leqslant x} f(n).$$ Then $$F(x+q)=F(x)+c,$$ for some constant $c$. The two choices for Dirichlet characters modulo $q$ yield two possibilities for $c$. If we are dealing with trivial Dirichlet character modulo $q$, then $c=q-1$, and if we are dealing with the quadratic character modulo $q$, then $c=0$. 

Notice that $F$ and $G$ satisfy the relationship $$G(x)=F(x)+f(q)F(x/q)+f(q)^2F(x/q^2)+\cdots,$$ where the sum ends once $q^i \geqslant x.$

If $c$ is nonzero, as stated above, we have $c=q-1$. Note that $$x\cdot\frac{q-1}{q} - q\leqslant F(x)\leqslant x\cdot\frac{q-1}{q}+ q,$$ for all $x$. 

Let $i$ be the first natural number such that $q^i \geqslant x$, and suppose that $x$ is large enough so that $i\geqslant 2.$ Then \begin{align*} G(x) &\geqslant x\cdot\frac{q-1}{q}-q-x\cdot\frac{q-1}{q^2}-q+x\cdot\frac{q-1}{q^3}-q - \cdots +(-1)^{i+1}x\cdot\frac{q-1}{q^i} -q,\\
&\geqslant x(q-1)\left(\frac{1}{q}-\frac{1}{q^2}+\frac{1}{q^3}-\cdots\right) - q(1+\log_q(x))\\
&>x\cdot\frac{(q-1)^2}{q^2} - q(1+\log_q(x)) 
\end{align*} which goes to infinity faster than $\log x$ for any $q\geqslant 2$.

If $c$ is zero, then the function $f$ agrees with the quadratic character modulo the odd prime $q$. Again we will argue using $$G(x)=F(x)+f(q)F(x/q)+f(q)^2F(x/q^2)+\cdots,$$ but now with the additional property that $F(x+q)=F(x)$. Indeed, this additional property gives that $F(qn)=0$ for all $n$. Now since $f$ is multiplicative, we necessarily have $f(1)=F(1)=1$, and so also, $F(qn+1)=1$, since $F(x)$ is periodic.

For this case, consider the number $$[\underbrace{1010101010\cdots 10}_{10\ written\ m\ times}]_{q} = [(10)^m]_{q}.$$ Then \begin{align*}G( [(10)^m]_{q})&= F( [(10)^m]_{q})+f(q)F([(10)^{m-1}1]_q)+f(q)^2F([(10)^{m-1}]_q)+\cdots\\
&=f(q)\sum_{i=0}^{m-1} F([(10)^{m-1}1]_q)\\
&=f(q)m,
\end{align*} and so $|G(x)|$ grows at least logarithmically.

Since $c$ is zero in this case, we have $|F(x)|\leqslant q,$ so that also $|G(x)|\leqslant q(1+\log_q x).$

The above argument allows us to finish with the following result, which answers a question of Borwein, Choi, and Coons \cite[Question 3]{BCC2010}.

\begin{proposition} Let $f(n)=\pm 1$ be a completely multiplicative function. If $f$ is $k$-automatic for some positive integer $k$, then there is a $C>0$ so that $$\max_{N\leqslant x}\Bigg|\sum_{n\leqslant N} f(n)\Bigg|>C\log x.$$ Moreover, if $\sum_{n\leqslant x} f(x)=o(x),$ then $$\max_{N\leqslant x}\Bigg|\sum_{n\leqslant N} f(n)\Bigg|\asymp \log x.$$
\end{proposition}

\bibliographystyle{srtnumbered}

\end{document}